\documentclass{amsart}
\usepackage{amssymb, amsmath, latexsym, a4,harmony}
\usepackage[latin1]{inputenc}
\usepackage[2cell,arrow,curve,matrix]{xy}
\usepackage{a4wide, amssymb, bbm, calligra, mathrsfs}
\UseHalfTwocells \UseTwocells

\def\xto#1{\xrightarrow[]{#1}}

\def\q{{\frak q}}

\def \hom{\mathop{\sf Hom}\nolimits}
\def \F{\mathop{\mathcal F}\nolimits}
\def \F{\mathop{\mathcal F}\nolimits}
\def \G{\mathop{\mathcal G}\nolimits}

\def\Grp{\bf Gpd }

\numberwithin{equation}{section}

\def \colim {\mathop{\sf colim}\nolimits}
\def \tl{\mathop{\sf 2lim}\nolimits}
\def \tc {\mathop{\sf 2colim}\nolimits}

\def\1{^{-1}}
\newtheorem{De}{Definition}[section]
\newtheorem{Th}[De]{Theorem}
\newtheorem{Pro}[De]{Proposition}
\newtheorem{Le}[De]{Lemma}
\newtheorem{Co}[De]{Corollary}

\newtheorem{Rem}[De]{Remark}

\newtheorem{Ex}[De]{Example}

\begin{document}

\title{Fundamental groupoid as a  terminal costack}

\author[I. Pirashvili]{Ilia  Pirashvili}
\address{
Department of Mathematics\\
University of Leicester\\
University Road\\
Leicester\\
LE1 7RH, UK} \email{ip96@leicester.ac.uk}
\maketitle

\begin{abstract} Let $X$ be a topological space. We denote by $\pi_0(X)$ the set of connected components of $X$ and by $\Pi_1(U)$ the fundamental groupoid. In this paper we prove that for good topological spaces the assignments $U\mapsto\pi_0(U)$ and $U\mapsto\Pi_1(U)$ are the terminal cosheaf and costack respectively. 
\end{abstract}

\bigskip
\section{Introduction}

For quit some time now, stack theory, which is the 2-mathematical analogue of sheaf theory, has enjoyed a significant amount of popularity, especially in the last decade. But despite the fact that sheaves and stacks have been studied deeply for so long, the very closely related cosheaves and costacks have been virtually ignored. In this paper we aim to demonstrate their importance and usefulness. We will show that for good topological spaces the assignment $U\mapsto \pi_0(U)$ is the terminal cosheaf of sets, and the assignment $U\mapsto \Pi_1(U)$ is the terminal costack of groupoids. 

Though the main result of this paper is conceptual rather then calculatory, Example. \ref{ex} demonstrates the advantage of 2-mathematics. Using 2-colimits rather then colimits, we are able to use Van Kampens theorem to calculate the fundamental groupoid of $S^1$ significantly easier.

This alone is a rather nice fact but it could have other important implications. This result generalises in two directions which will be in forthcomming papers. The first one is rather obvious, that the fundamental $n$-groupoid would most likely be the terminal co-$n$-stack with values in $n$-groupoids. The other is that if we define the fundamental groupoid to be the terminal costack, we don't need notions like paths and coverings. Indeed since the terminal costack (resp. terminal cosheaf) was the costackification (resp. cosheafification) of the constant functor with values in the one point category (resp. the one point set), Thm. \ref{vo} gives hope that we could define the fundamental groupoid for other more pathological spaces, and indeed for general 2-sites as well. 

This paper is part of my PhD thesis at the University of Leicester under the supervision of Dr. Frank Neumann. He introduced me to the fundamental groupoid and insipred much of this paper, for which I would like to thank him.

\section{$\pi_0$ as a terminal cosheaf}

Let $X$ be a topological space. Denote by $\mathfrak{Off}(X)$ the poset of open subsets of $X$. 

\begin{De}\label{cosh} Let $G$ be a (covariant) functor from the poset $\mathfrak{Off}(X)$ to a category $\bf A$. We say that $G$ is a cosheaf, if for any open covering  $U=\bigcup_iU_i$ the diagram
$$\xymatrix{\coprod_{i,j}G(U_i\cap U_j) \ar@<.5ex>[r]^{ \ \ \ \ \ b_0} \ar@<-.5ex>[r]_{ \ \ \ \ \ b_1} & \coprod_i G(U_i) \ar[r]^{ \ \ a} & G(U) \ar[r] & 1 }$$
is exact, in other words,  
$G(U)$ is the coequiliser of $b_0,b_1$. That is to say, $G$ is a cosheaf if and only if for any object $S$ the functor $F$ defined by $F(U)= \hom_{\bf A}(G(U),S)$ is a sheaf of sets \cite{br}.
\end{De} 

Let $G:\mathfrak{Off}(X)\rightarrow {\bf A}$ be a functor. We say that $\hat{G}:\mathfrak{Off}(X)\rightarrow {\bf A}$ is the associated cosheaf of $G$ if there is given a natural transformation  $\epsilon:\hat{G}\rightarrow G$, and for all cosheafs $G':\mathfrak{Off}(X)\rightarrow {\bf A}$ and all natural transformations $\varphi:G'\rightarrow G$, there is a unique natural transformation $\hat{\varphi}:G'\rightarrow \hat{G}$ such that the diagram
$$\xymatrix{G'\ar[r]^{\varphi}\ar[dr]_{\hat{\varphi}} & G \\
	           & \hat{G}\ar[u]_{\epsilon}}$$
commutes. 

Observe the following theorem    \cite[Theorem  6.28] {jajr}.

\begin{Th}\label{vo} Assuming Vop\v{e}nka's  principle, every full subcategory of a locally presentable category $\bf K$, closed under colimits in $\bf K$, is coreflective.   
\end{Th}

As a corollary of this theorem, we have immediately that, assuming Vop\v{e}nka's  principle, cosheafification exists whenever $G$ has values in sets (or many other categories, like groups etc). However, we do not make use of this fact. We are mainly interested in the terminal cosheaf, which is easy  to construct under the assumption that $X$ is locally connected. For a topological space $X$ we let $\pi_0(X)$ be the set of all connected components of $X$.

\begin{Le} Let $X$ be a locally connected topological space. Then $U\mapsto \pi_0(U)$ is a cosheaf.
\end{Le}

\begin{proof} To prove this, we will show that $U\mapsto \hom(\pi_0(U),S)$ is a sheaf for all $S$. But whenever $X$ is locally connected, $\hom(\pi_0(U),S)$ is the set of all locally constant functions on $U$ (LCF(U) for short) with values in $S$. Hence $U\mapsto LCF(U)$ is clearly a sheaf, which implies the lemma.
\end{proof}

\begin{Pro} Let $X$ be a locally connected topological space and consider the constant functor $p$, where $p(U)=1$, the one point set for any open set $U$. Then we have $\hat{p}(U)=\pi_0(U)$ where $\hat{p}$ is the cosheafification of $p$.
\end{Pro}

\begin{proof} The proof is done by the universality condition of the associated cosheaf. For this, we first need to construct a map $\alpha$ from $G$ to $\hat{p}$ for every cosheaf $G$ defined on $X$. Then we need to check that it is unique. To define $G(U)\rightarrow \hat{p}(U)$, consider the cover $\{U_i\}$ of $U$ by its connected components. Now it is clear that we have maps 
$$\xymatrix{\coprod_{i,j}G(U_i\cap U_j) \ar@<.5ex>[r]^{ \ \ \ \ \ b_0} \ar@<-.5ex>[r]_{ \ \ \ \ \ b_1}\ar[d] & \coprod_i G(U_i) \ar[r]^{ \ \ a}\ar[d] & G(U) \ar[r] & 1 \\ 
	          \coprod_{i,j}\pi_0(U_i\cap U_j) \ar@<.5ex>[r]^{ \ \ \ \ \ b_0} \ar@<-.5ex>[r]_{ \ \ \ \ \ b_1} & \coprod_i \pi_0(U_i) \ar[r]^{ \ \ a} & \pi_0(U) \ar[r] & 1 }$$ 
since the intersection of connected components is either itself or empty, and $\pi_0$ of a connected space is 1, the one point set. It follows from the exactness of the sequences that there is a unique map $G(U)\rightarrow \pi_0(U)$. Since every map in the one point set is unique, the middle map is unique, which implies the uniqueness of $\alpha$. This finishes the proof.
\end{proof}

\begin{Co} The terminal object of the category of cosheafs over a locally connected topological space $X$ is given by $U\mapsto \pi_0(U)$. 
\end{Co}

\section{$\Pi_1$ and Costacks}
Recall that a \emph{groupoid} is a category where all morphisms are isomorphisms. A groupoid $\sf G$ is \emph{simply connected} \cite[p. 36]{gz} provided for any two objects $a$ and $b$ of $\sf G$, there is exactly one morphism $a\to b$. By abuse of notations we let $\Grp$ be the category, as well as the $2$-category of small groupoids. Thus for any functor $\F:I\to \Grp$ we can talk not only of limits and colimits of $\F$, but also of 2-limits and 2-colimits (see Section \ref{2lim}) of $\F$. 

\begin{De}\label{cost} Let $\G$ be a covariant functor from  the poset $\mathfrak{Off}(X)$ to $\Grp$. We say that $\G$ is a costack, if 
$$\xymatrix{\coprod\limits_{i,j,k}\G(U_i\cap U_j\cap U_k) \ar@<.7ex>[r]\ar[r]\ar@<-.7ex>[r]&\coprod\limits_{i,j}\G(U_i\cap U_j) \ar@<.5ex>[r]\ar@<-.5ex>[r] & \coprod\limits_i \G(U_i) \ar[r] & \G(U) \ar[r] & 0 }$$
is $2$-exact. That is to say, $\G(U)$ is the $2$-colimit of the diagram 
$$\xymatrix{\coprod\limits_{i,j,k}\G(U_i\cap U_j\cap U_k) \ar@<.7ex>[r]\ar[r]\ar@<-.7ex>[r]&\coprod\limits_{i,j}\G(U_i\cap U_j) \ar@<.5ex>[r]\ar@<-.5ex>[r] & \coprod\limits_i \G(U_i)  &  &  }.$$
Equivalently, 
$\G$ is a costack iff  for any groupoid ${\sf G}$ the assignment $$U\mapsto \hom_{\Grp}(\G(U),{\sf G})$$ defines a stack \cite{br}.
\end{De}

Hence we can again talk about costackification, or associated costacks of functors, if they exist. Probably an analogue of Theorem \ref{vo} is still true for $2$-categories and hence under some set theoretical assumptions the costackification does exist, but we do not go in this direction and we restrict ourselves  only to the existence of 2-terminal objects. Recall that an object $t$ of a $2$-category $\bf A$ is  \emph{$2$-terminal}, provided for any object $x$ of $\bf A$ the category $\hom_{\bf A}(x,t)$ is a simply connected groupoid, which is the same to say that $\hom_{\bf A}(x,t)$ is equivalent to the groupoid $\bf 1$, with one object and one arrow.

For a topological space $X$, we let $\Pi_1(X)$ be the fundamental groupoid of $X$, see \cite{Brown} . Recall that the objects of $\Pi_1(X)$ are points of $X$, while morphisms are homotopy classes of paths. Then we have the following result.

\begin{Th}\label{Pi1} Let $X$ be a  topological space. Then the association $U\mapsto \Pi_1(U)$ defines simultaneously a cosheaf and costack with values in $\Grp$. 
\end{Th}

The statement on cosheaves is basically the classical Van-Kampen theorem formulated in terms of fundamental groupoids, compare with \cite[Theorem II.7] {May}, where the author considers only coverings with connected open sets.  The last condition is unnecessary,  as it was demonstrated in \cite[p. 226, Statement 6.7.2]{Brown} where the author considers only coverings with two members. The statement about costacks is the main part of the paper and  the proof will be given in Section \ref{pr22}.

This already demonstrates the importance of costacks, but indeed more is true. Call a topological space $X$ \emph{good} if any open subset $U$ of $X$ possesses an open covering $U=\bigcup_iU_i$ such that all $U_i$ are simply connected, as well as all nonempty intersections $U_i\cap U_j$, $U_i\cap U_j\cap U_k$.

\begin{Th} Let $X$ be a good topological space. Consider the constant functor $P:U\mapsto {\bf 1}$, where ${\bf 1}$ is the one point groupoid (with only one object and one morphism). Then $P$ has an associated costack, which is the fundamental groupoid, i.e. $\hat{P}(U)=\Pi_1(U)$. Thus the costack  given by $U\mapsto \Pi_1(U)$ is a $2$-terminal object in the 2-category of all costacks on $X$.
\end{Th}

\begin{proof} By Theorem \ref{Pi1} $U\mapsto \Pi_1(U)$ is a costack. To prove that it is the costackification of $P$, we will use the universal property. It is clear that we have a morphism $\mathfrak{p}:\Pi_1\rightarrow P$ of 2-functors, so we only have to prove that given another 2-functor $Q:U\rightarrow {\Grp}$, there exists a morphism $\mathfrak{q}:\Pi_1\rightarrow Q$, where $Q$ is a costack. The fact that the compatibility conditions will be satisfied comes from the triviality of $P$. 

Hence we only have to construct $\mathfrak{q}$. To define $\mathfrak{q}(U)$, cover $U$ by open subsets $\{U_i\}_{i\in I}$ such that all $U_i$ are simply connected  and the same holds for all their nonempty pairwise and triple intersections. We can do this by the assumption on $X$. Since $Q$ is a costack

$$\xymatrix{\coprod\limits_{i,j,k}Q(U_i\cap U_j\cap U_k) \ar@<.7ex>[r]\ar[r]\ar@<-.7ex>[r]&\coprod\limits_{i,j}Q(U_i\cap U_j) \ar@<.5ex>[r]\ar@<-.5ex>[r] & \coprod\limits_i Q(U_i) \ar[r] & Q(U) \ar[r] & 1 }$$
is $2$-exact. Likewise for $\Pi_1$. Note that on the empty set $Q$ is empty, just like $\Pi_1$. For $U_i$-s and their non-empty intersections $U_i\cap U_j$ and $U_i\cap U_j\cap U_k$, $\Pi_1$ is the trivial groupoid, since they are simply connected. It follows that there is essentially a unique functor from $Q(U_i), Q(U_i\cap U_j)$ and $Q(U_i\cap U_j\cap U_k)$ to $\Pi_1(U_i), \Pi_1(U_i\cap U_j)$ and $\Pi_1(I_i\cap U_j\cap U_k)$ respectively.

Hence by $2$-exactness we get a map, unique up to natural transformation $\q(U):Q(U)\rightarrow \Pi_1(U)$ satisfying the compatibility condition. This finishes the proof.
\end{proof}

\section{Limits, colimits, $2$-limits and 2-colimts }\label{2lim}
First let us fix some notations. For functors and natural transformations
$$\xymatrix{{\sf A}\ar[r]^F&{\sf B}\ar@<.5ex>[r]^{ \ \ \ \ \ G_1} \ar@<-.5ex>[r]_{ \ \ \ \ \ G_2} & {\sf C}\ar[r]^T &{\sf D}}, \ \  \alpha:G_1 \Longrightarrow G_2. $$
one denotes by $\alpha \star F$ and  $T\star \alpha$ the
induced natural transformations $G_1F\Longrightarrow G_2F$, $TG_1\Longrightarrow TG_2$.

Let $I$ be a poset and let $\F:I^{op}\to \Grp$ be  a contravariant functor from the category $I$ to the category of groupoids. For an element $i\in I$ we let $\F_i$ be the value of $\F$ at $i$. For $i\leq j$ we let $\psi_{ij}:\F_j\to \F_i$ be the induced functor. Recall the construction of two groupoids $\lim\limits_{i} \F_i$ and $\tl\limits_i\F_i$, called respectively \emph{limit} and \emph{$2$-limit} of $\F$. 

Objects of the groupoid $\lim\limits_i \F_i$ (or simply $\lim \F$) are  families $(x_i)$, where $x_i$ is an object of the category $\F_i$ such that for any $i\leq j$ one has $\psi_{ij}(x_j)=x_i$. A morphism $(x_i)\to(y_i)$ is  a family $(f_i)$, where $f_i:x_i\to y_i$ is a morphism of $\F_i$ such that for any $i\leq j$ one has $\psi_{ij}(f_j)=f_i$.
 
 Objects of the groupoid $\tl\limits_i\F_i$ (or simply $\tl \F$) are  collections $(x_i,\xi_{ij})$, where $x_i$ is an object of $\F_i$, while  $\xi_{ij}:\psi_{ij}(x_j)\to x_i$ for  $i\leq j$ is an isomorphism of the category $\F_j$ satisfying the 1-cocycle condition: For any $i\leq j\leq k$ one has
 $$\xi_{ik}=\xi_{ij} \circ \psi_{ij}(\xi_{kj}).$$
A morphism from $(x_i,\xi_{ij})$ to $(y_i,\eta_{ij})$ is  a collection $(f_i)$, where $f_i:x_i\to y_i$ is a morphism of $\F_i$ such that for any $i\leq j$ the following 
$$\xymatrix{\psi_{ij}(x_j)\ar[r]^{\xi_{ij}}\ar[d]_{\psi_{ij}(f_j)}& x_i\ar[d]^{f_i}\\
\psi_{ij}(y_j)\ar[r]_{\eta_{ij}}& y_i}$$
 is a commutative diagram.  Let us state the following obvious Lemma.
 \begin{Le}\label{gamma} The functor $$\gamma:\lim \F\to \tl \F,$$ given by 
 $$\gamma(x_i)=(x_i,Id_{x_i})$$
 is  full and faithful.
 \end{Le}

 Dually, for  a functor $\G:I\to \Grp$ one can define \emph{colimits} and \emph{$2$-colimits} of $\G$ as follows. $\colim\limits_i \G_i$ is a groupoid together with a collection of functors 
$$\alpha_i:\G_i\to \colim\limits_i \G_i$$
 such that $\alpha_i=\alpha_j\psi_{ij}$ for any $i\leq j$. Additionally, one requires that for any groupoid ${\sf G}$ the canonical functor
$$c:\hom_{\Grp}(\colim\limits_i\G_i,{\sf G})\to \lim\limits_i\hom_{\Grp}(\G_i,{\sf G})$$
is an isomorphism of groupoids. Here the functor $c$ is given by $c(\chi)=(\chi\circ  \alpha_i)$,  where $\chi:\colim\limits_i\G_i\to \sf G$ is a functor. It is well-known that such $\colim\limits_i \G_i$ exists and it is unique up to isomorphism of groupoids \cite[p. 36]{gz}.

Quite similarly, $\tc\limits_i \G_i$ is a groupoid together with a family of functors 
 $$\alpha_i:\G_i\to \tc\limits_i \G_i$$ 
 and natural transformations $\lambda_{ij}:\alpha_j\psi_{ij}\Longrightarrow \alpha_i$, $i\leq j$ 
 satisfying  the $1$-cocycle condition: For any $i\leq j\leq k$, 
 $$\lambda_{ik}=\lambda_{ij}\circ (\lambda_{jk}\star\psi_{ij}).$$
Furthermore, one requires that for any groupoid ${\sf G}$, the canonical functor
$$\kappa:\hom_{\Grp}(\tc\limits_i\G_i,{\sf G})\to \tl\limits_i\hom_{\Grp}(\G_i,{\sf G})$$
is an equivalence of groupoids. Here the functor $\kappa$ is given by $\kappa(\chi)=(\chi\circ  \alpha_i, \chi_i\star \lambda_{ij})$. It is well-known that $\tc$ exists and is unique up to equivalence of groupoids, see, \cite[pp. 192-193] {br}. 

The functors $c, \kappa$  and the comparison functor $\gamma$ in Lemma \ref{gamma} allow us to consider the composite functor:
$$\hom_{\Grp}(\colim\limits_i\G_i,{\sf G})\xto{c} \lim\limits_i\hom_{\Grp}(\G_i,{\sf G})\xto{\gamma} \tl\limits_i\hom_{\Grp}(\G_i,{\sf G}) \xto{\kappa^{-1}} \hom_{\Grp}(\tc\limits_i\G_i,{\sf G})$$
By the Yoneda lemma one obtains a  comparison functor
$$\delta:\tc\limits_i\G_i\to \colim\limits_i\G_i.$$
\begin{Le}\label{delta} The functor $\delta$ is an equivalence of categories iff for any groupoid ${\sf D}$, and any family of functors 
 $$\beta_i:\G_i\to {\sf D}$$ 
 and natural transformations $\tau_{ij}:\beta_j\psi_{ij}\Longrightarrow \beta_i$, $i\leq j$ 
 satisfying  the $1$-cocycle condition [i.e. for any $i\leq j\leq k$, 
 $$\tau_{ik}=\tau_{ij}\circ (\tau_{jk}\star\psi_{ij})],$$
 there exists another family of functors $\beta_i':\G_i\to {\sf D}$
 and natural transformations $\mu_i:\beta_i\Longrightarrow \beta'_i$ such that
 $\beta_i'=\beta'_j\circ \psi_{ij}$,  and $\tau_{ij}=\mu_i^{-1}\circ (\mu_j\star \psi_{ij})$ for all $i\leq j$.
\end{Le}
\begin{proof} By the Yoneda lemma for $2$-categories $\delta$ is an equivalence of categories iff for any groupoid ${\sf D}$ the induced functor 
$$\hom_{\Grp}(\delta, {\sf D}):\hom_{\Grp}(\colim \G,{\sf D})\to \hom_{\Grp}(\tc \G,{\sf D})$$ is an equivalence of categories. Since
$$\hom_{\Grp}(\colim\limits_i  \G,{\sf D}) =\lim\limits_i \hom_{\Grp}(\G_i,{\sf D})$$ 
and 
$$\hom_{\Grp}(\tc\limits_i  \G,{\sf D}) =\tl\limits_i \hom_{\Grp}(\G_i,{\sf D}),$$ 
we can use Lemma \ref{gamma} to conclude that the functor $\hom_{\Grp}(\delta, {\sf D})$ is always full and faithful. Hence it is an equivalence iff it is essentially surjective on objects. One easily sees that the last condition is exactly the assertion made in the statement and we are done.

\end{proof}
As we will see in the next section, the functor $\delta$ is often an equivalence of categories, which explains the fact that $U\mapsto \Pi_1(U)$ is simultaneously a cosheaf and a costack.

\section{When $\delta:\tc\to \colim$ is an equivalence of categories}
Recall that a poset $I$ is called \emph{filtered} provided for any elements $i,j\in I$ there is an element $k\in I$ such that $i\leq k$ and $j\leq k$.

\begin{Pro}\label{filtered} If $I$ is  a filtered poset, then for any functor $\G:I\to \Grp$ the canonical comparison functor $\delta:\tc\limits_i\G_i\to \colim\limits_i\G_i $ is an equivalence of categories.
\end{Pro}
\begin{proof}  In this case the construction of $\colim\limits_i\G_i $ given in \cite[p.36]{gz} has the following description. 
Let  $X$ and $Y$ be objects of the  categories $\G_i$ and $\G_j$ respectively. We will say that $X$ and $Y$ are equivalent, if there exists $k\in I$, such that  $i\leq k,j\leq k$  and $\psi_{ik}(X)=\psi_{jk}(Y).$ The equivalence class of $X$ is denoted by $\{X\}$. The collection of such equivalence classes forms the class of all objects of the category $\colim\limits_i\G_i $.  If $\{X\}$ and $\{Y\}$ are two equivalence classes, assuming $X\in \G_i$ and $Y\in \G_j$, then a morphism $\{X\}\to \{Y\}$ is the equivalence class of morphisms $f:\psi_{ik}(X)\to \psi_{jk}(Y)$ of the category $\G_k$, where $i\leq k, j\leq k$. Here two morphisms $f:\psi_{ik}(X)\to \psi_{jk}(Y)$ and $g:\psi_{im}(X)\to \psi_{jm}(Y)$, $i\leq m, j\leq m$ are equivalent if there exists an element $n\in I$ such that $k\leq n, m\leq n$ and $\psi_{kn}(f)=\psi_{mn}(g).$

Recall that the construction of $\tc \G$ given in  \cite[pp. 192-193] {br} shows that the objects of $\tc \G$ are a disjoint union of the objects of $\G_i$, $i\in I$, while for $X\in \G_i$ and $Y\in \G_j$  one has  
$$\hom_{\tc\limits_i \G_i}(X,Y)=\hom_{\colim\limits_i \G_i}(\{X\},\{Y,\}).$$
It follows that the functor $\delta$, which sends $X$ to $\{X\}$ is obviously full and faithful and surjective on objects. Hence $\delta$ is an equivalence of categories.
\end{proof}
\begin{Pro}\label{2} Let  
$$\xymatrix{ {\bf A}\ar[r]^{i_1}\ar[d]_{i_2} & {\bf B} \\ 
	         {\bf C} & }$$
be a diagram of groupoids, where $i_1$ is injective on objects. Then the colimit and 2-colimit are equivalent. 
\end{Pro} 
Thanks to Lemma \ref{delta} the result follows  from the following deformation lemma.
\begin{Le} \label{deformation} Let $$\xymatrix{ {\bf A} \drtwocell\omit{\lambda}\ar[r]^{i_1}\ar[d]_{i_2} & {\bf B}\ar[d]^{j_1} &  \\ {\bf C}\ar[r]_{j_2} & {\bf D} &  }$$
be a $2$-diagram of groupoids, i.e. $\lambda:j_1i_1\Longrightarrow j_2i_2$ is a natural transformation.
If the functor $i_1$ is injective on objects, then there exists a functor $j_1':{\bf B}\to {\bf D}$ and a natural transformation $\kappa:j_1\Longrightarrow j_1'$ for which the diagram 
$$\xymatrix{ {\bf A}\ar[r]^{i_1}\ar[d]_{i_2} & {\bf B}\ar[d]^{j_1'} \\ {\bf C} \ar[r]_{j_2}& {\bf D}}$$
commutes and $\lambda$ coincides with $\kappa\star i_1:j_1i_1\Longrightarrow j_1'i_1=j_2i_2.$
\end{Le}
\begin{proof} Define $j_1'$ on objects by
$$j_1'(b)=\begin{cases}j_1(b), & {\rm if} \ b\not\in Im(i_1)\\ j_2i_2(a),& {\rm if} \ b=i_1(a).\end{cases}$$
To define $j_1'$ on morphisms we proceed as follows.
Let $\beta:b_1\to b_2$ be a morphism in $\bf B$. To define $$j_1'(\beta):j_1'(b_1)\to j_1'(b_2)$$ we have to consider five different cases.

{\it Case 1}. $b_1=i_1(a_1)$ and $b_2\not\in Im(i_1)$. One defines $j_1'(\beta)$ to be the composite:
$$j_1'(b_1)=j_2i_2(a_1)\xto{\lambda^{-1}(a_1)}j_1i_1(a_1)=j_1(b_1)\xto{j_1(\beta)} j_1(b_2)=j_1'(b_2).$$

{\it Case 2}. $b_1=i_1(a_1)$, $b_2=  i_1(a_2)$, but $\beta \not \in Im(i_1)$. One defines $j_1'(\beta)$ to be the composite:
$$j_1'(b_1)=j_2i_2(a_1)\xto{\lambda^{-1}(a_1)}j_1i_1(a_1)=j_1(b_1)\xto{j_1(\beta)} j_1(b_2)=j_1i_1(a_2)\xto{\lambda(a_2)} j_2i_2(a_2)=j_1'(b_2).$$

{\it Case 3}. $\beta\in Im(j_1)$. We choose $\alpha:a_1\to a_2$ in $\bf A$ such that $\beta=i_1(\alpha)$. One defines $j_1'(\beta)$ by
$$j_1'(b_1)=j_2i_2(a_1)\xto{j_2i_2(\alpha)} j_2i_2(a_2)=j_1'(b_2).$$
To check that this is independent on the choice of $\alpha$, let us also consider also $\alpha':a_1\to a_2$ with the property $i_1(\alpha')=\beta$. Since $\lambda$ is a natural transformation, we have a commutative diagram
$$\xymatrix{j_1i_1(a)\ar[rr]^{\lambda(a)}\ar@<-.5ex>[d]_{_{j_1i_1(\alpha)}}\ar@<.5ex>[d]^{_{j_1i_1(\alpha')}} & &  j_2i_2(a)\ar@<-.5ex>[d]_{_{j_2i_2(\alpha)}}\ar@<.5ex>[d]^{_{j_2i_2(\alpha')}} \\
	          j_1i_1(a')\ar[rr]_{\lambda(a')} & & j_2i_2(a') }$$
Since the left vertical arrows are equal and horizontal ones are isomorphisms, it follows that the right vertical arrows are also equal. Hence $j_1'(\beta)$ is well-defined in this case.

{\it Case 4}. $b_1\not\in Im(i_1)$ and $b_2=i_1(a_2)$. One defines $j_1'(\beta)$ to be the composite:
$$j_1'(b_1)=j_1(b_1)\xto{j_1(\beta)} j_1(b_2)=j_1i_i(a_2)\xto{\lambda(a_2)} j_2i_2(a_2)=j_2'(b_2)$$

{\it Case 5}. $b_1\not\in Im(i_1)$  and $b_2\not\in Im(i_1)$. One defines $j_1'(\beta)$ as
$$j_1'(b_1)=j_1(b_1)\xto{j_1(\beta)} j_1(b_2)=j_1'(b_2)$$
Checking case by case shows that $j_1'$ is really a functor, with $j_1'i_1=j_2i_2$. Define $\kappa$  by
$$\kappa(b)=\begin{cases} Id_{j_1(b)}, & {\rm if} \ b\not\in Im(i_1)\\ \lambda(a),& {\rm if} \ b=i_1(a). \end{cases} $$
One easily sees that $j_1'$ and $\kappa$ satisfy the assertions of the Lemma.

\end{proof}

It follows that the classical  Van Kampen theorem  \cite[p. 226]{Brown} can be restated as follows.

\begin{Th} \label{vk} (Van Kampen) Let $U$ and $V$ be open subsets of $X$. Assume $X= U \cup V$. Then the diagram is simultaneously a pushout and $2$-pushout of groupoids:
$$\xymatrix{\Pi_1(W) \ar[d]\ar[r] & \Pi_1(U)\ar[d]\\ \Pi_1(V)\ar[r] & \Pi_1(X),}$$
where  $W=U \cap V$. 
\end{Th}

The advantage of  $2$-pushouts is that the fundamental groupoids involved can be changed by equivalent ones. Here is an example to demonstrate the usefulness of $2$-pushouts.

\begin{Ex} \label{ex}  Let us cover $S^1$ by $U$ and $V$, where both are open subsets of $S^1$, with a single point removed (of course different). Since $\Pi_1(U)$ and $\Pi_1(V)$ are equivalent to the groupoid ${\bf 1}$ with one object and one morphism, and $\Pi_1(W)$ is equivalent to the groupoid
${\bf 1} \coprod {\bf 1}$, we see that $\Pi_1(S^1)$ is equivalent to the $2$-pushout (and hence pushout) of the diagram
$$\xymatrix{{\bf 1} \coprod {\bf 1} \ar[r] \ar[d] & {\bf 2} \\ {\bf 1} &}$$
where $\bf 2$ is the simply connected groupoid with two objects and the horizontal arrow is bijective on objects. Clearly the pushout of this diagram is the groupoid with one object and the infinite cyclic group as the group of automorphisms. Compare this computation with the computation in \cite[pp. 233-234] {Brown} which uses the more complicated groupoid $\Pi_1XA$, where $A$ is an additional set. 
\end{Ex}

\section{Proof of Theorem \ref{Pi1}}\label{pr22}
Let $X$ be a topological space and $\F:\mathfrak{Off}(X)^{op}\rightarrow \mathsf{Grp}$ be a functor. 

\begin{De}
Let $n\geq 2$ be an integer. We will say that $\F$ satisfies the property ${\mathsf sh}(n)$ (resp. ${\mathsf st(n)})$ if for any open subset $U$ of $X$ and any open cover $U_1,\cdots, U_n$ of $U$, the natural functor 
$$\xymatrix{\F(U)\ar[r] & \lim [\prod_i \F(U_i)\ar@<.5ex>[r]\ar@<-.5ex>[r] & \prod_{i,j}\F(U_i\cap U_j)]}$$
is an isomorphism of categories, (resp.
$$\xymatrix{\F(U)\ar[r] & \tl [\prod_i \F(U_i)\ar@<.5ex>[r]\ar@<-.5ex>[r] & \prod_{i,j}\F(U_i\cap U_j)\ar@<.7ex>[r]\ar[r]\ar@<-.7ex>[r] & \prod\limits_{i,j,k}\F(U_i\cap U_j\cap U_k)]}$$ 
is an equivalence of categories). That is to say, $\F$ satisfies the sheaf (resp. stack) condition for all coverings having $n$-members.
\end{De}

The aim of this section is to prove Theorem \ref{Pi1}, which will be done with formal arguments. In Thm. \ref{vk} we have shown that $U\rightarrow \hom_{\Grp}(\Pi_1(U),{\sf G})$ satisfies the $sh(2)$ as well as the $st(2)$ property for any groupoid $G$. In order to show that it is a sheaf (resp. a stack), we have to show that these conditions hold for general coverings, see Definition \ref{cosh} and Definition \ref{cost}. To do so, we will first show that $sh(2)$ (resp. $st(2)$) implies $sh(n)$ (resp. $st(n)$) for any $n\in\mathbb{N}$ and then use \ref{filtered} to prove the general case. 

\begin{Le} \label{finite} If $\F$ satisfies the condition ${\mathsf sh}(2)$, (resp. ${\mathsf st}(2)$) then $\F$ satisfies the condition ${\mathsf sh(n)}$ (resp. ${\mathsf st}(n)$) for any $n\geq 2$. 
\end{Le}

\begin{proof} We consider only the case $n=3$, since the only difference between this and the general case is notation. By definition, the objects of 
$$\xymatrix{\tl[\prod_i \F(U_i)\ar@<.5ex>[r]\ar@<-.5ex>[r] & \prod_{i,j}\F(U_i\times_UU_j)\ar@<.7ex>[r]\ar[r]\ar@<-.7ex>[r] & \prod\limits_{i,j,k}\F(U_i\times_UU_j\times_UU_k)]}$$ 
are $$((g_1,g_2,g_3),\alpha_{12}, \alpha_{13},\alpha_{23})$$
where $g_i$ is an object of $\F(U_i)$, $i=1,2,3$ and 
$\alpha_{12}:g_1|_{U_{12}}\rightarrow g_2|_{U_{12}}$, $\alpha_{13}:g_1|_{U_{13}}\rightarrow g_3|_{U_{13}}$  and $\alpha_{23}:g_2|_{U_{23}}\rightarrow g_3|_{U_{23}}$ are morphisms respectively in $\F(U_{12})$, $\F(U_{13})$  and $\F(U_{23})$. Here $U_{ij}=U_i\cap U_j$. One also requires that these data satisfy
the 1-cocylce condition. The morphisms are $(h_1:g_1\rightarrow g'_1, h_2:g_2\rightarrow g'_2, h_3:g_3\rightarrow g'_3)$, such that restricted on the intersections, the obvious diagrams commute.

We set $V=U_1\cup U_2$. Since $\F$ satisfies the condition $st(2)$, we can assume that the objects of $F(U)=F(V \cup U_3)$ are
the same as of the category 
$$\xymatrix{\tl[\F(V)\times\F(U_3) \ar@<.5ex>[r]\ar@<-.5ex>[r] & \F(V\cap U_3)]}$$
So we can assume that they are of the form 
$$\{((g_v,g_3),\gamma:g_v|_{V\cap U_3}\rightarrow g_3|_{V\cap U_3})\},$$
where  $g_v$ is an object of $\F(V)$. But since $V=U_1\cup U_2$ and $\F$ satisfies the condition $st(2)$ we may assume that the objects of $\F(V)$ have the form 
$$g_v=\{((g_1,g_2),\delta:g_1|_{U_{12}}\rightarrow g_2|_{U_{12}})\}.$$ 
Thus the objects  of $\F(V\cup U_3)$ are 
$$\{((g_v,g_3),\gamma:g_v|_{V\cap U_3}\rightarrow g_3|_{V\cap U_3})\}.$$ 
Hence the objects of $\F(U)=\F(V\cup U_3)$, can be seen as 
$$\{((g_1,g_2,g_3),\delta:g_1|_{U_{12}}\rightarrow g_2|_{U_{12}}, \gamma_1:g_1|_{U_{13}}\rightarrow g_3|_{U_{13}}, \gamma_2:g_2|_{U_{23}}\rightarrow g_3|_{U_{23}})\},$$
such that 
$$\xymatrix{ g_1|_{U_{123}}\ar[r]\ar[d]_{\delta|_{U_{123}}} & g_3|_{U_{123}}\ar[d]^{id} \\ 
	           g_2|_{U_{123}}\ar[r] & g_3|_{U_{123}} }$$commutes. 
	           Since the last condition  is exactly  the 1-cocycle condition, we can see that the categories $\F(U)$ and $\xymatrix{\tl[\prod_i \F(U_i)\ar@<.5ex>[r]\ar@<-.5ex>[r] & \prod_{i,j}\F(U_i\times_UU_j)\ar@<.7ex>[r]\ar[r]\ar@<-.7ex>[r] & \prod\limits_{i,j,k}\F(U_i\times_UU_j\times_UU_k)]}$ have essentially the same objects. Observe that the morphims of $\F(U)=\F(V\cup U_3)$ are $\{(\varphi_V:g_V\rightarrow g'_V, \varphi_{U_3}:g_3\rightarrow g_3')\}$, such that the obvious diagram commutes. Then morphisms of $\F(V)$ itself are $\{(\lambda_{U_1}:g_1\rightarrow g_1',\lambda_{U_2}:g_2\rightarrow g_2')\},$ such that again the obvious diagram commutes. Hence putting them together, we get $\{(\varphi_V:g_{U_1\cup U_2}\rightarrow g'_{U_1\cup U_2}, \varphi_{U_3}:g_3\rightarrow g_3')\}$ i.e. $\{(\varphi_{U_1}:g_1\rightarrow g'_1, \varphi_{U_2}:g_2\rightarrow g_2'),\varphi_{U_3}:g_3\rightarrow g_3')\}$ such that on the intersections, the restrictions agree. This shows the equivalence of $\F(U)$ and $\xymatrix{\tl[\prod_i \F(U_i)\ar@<.5ex>[r]\ar@<-.5ex>[r] & \prod_{i,j}\F(U_i\times_UU_j)\ar@<.7ex>[r]\ar[r]\ar@<-.7ex>[r] &\prod\limits_{i,j,k}\F(U_i\times_UU_j\times_UU_k)]}$. Hence $st(3)$ holds. 

The fact that ${\mathsf sh}(2)$ implies ${\mathsf sh}(3)$ is a special case of the above proof, by simply taking $\alpha_i$'s  to be identities.
\end{proof} 

To prove the general case, we basically only have to prove the following lemma: 

\begin{Le} \label{general} Let $\G:\mathfrak{Off}(X)\to {\Grp}$ be  a functor which commutes with filtered colimits. Then $\G$ is a cosheaf (resp. costack) if and only if for all groupoids ${\sf G}$, the functor $\F$ defined by $\F(U)=\hom_{Grp}(\G(U),{\sf G})$ satisfies the ${sh}(n)$ (resp.${st}(n)$) condition for all $n$. 
\end{Le}

\begin{proof} The 'only if' part is obvious. Assume $\F$ satisfies the condition ${\mathsf sh}(n)$ for all $n$.  We have to prove that $\G$ is a cosheaf. Let $U$ be an open set and $\{U_i\}_{i\in I}$ be a cover of $U$. Denote by $f(I)$ the set of finite subsets. We know that $f(I)$ is a filtered system for all $I$. For a fixed $\lambda\in f(I)$, denote by $U_\lambda$ the union of all $U_i$, with $i\in\lambda$. By assumption, $\G$ satisfies the cosheaf condition for any finite covering. Hence we have 
$$\xymatrix{\G(U)=\G(\colim\limits_{\lambda}U_{\lambda})=\colim\limits_{\lambda}\G(U_{\lambda})=
\colim\limits_{\lambda}(\colim[\coprod\limits_{i,j\in \lambda}\G(U_{ij}) \ar@<.5ex>[r]\ar@<-.5ex>[r] & \coprod\limits_{i\in\lambda}\G(U_i)])= }$$
$$\xymatrix{
\colim([\colim\limits_{\lambda}\coprod\limits_{i,j\in \lambda}\G(U_{ij}) \ar@<.5ex>[r]\ar@<-.5ex>[r] & \colim\limits_{\lambda}\coprod\limits_{i\in\lambda}\G(U_i)])=\colim[\coprod\limits_{i,j\in I}\G(U_{ij}) \ar@<.5ex>[r]\ar@<-.5ex>[r] & \coprod\limits_{i\in I}\G(U_i)] }$$
as desired. This shows that $G$ is a cosheaf.

If $\F$ satisfies the condition $st(n)$, then by Proposition \ref{filtered} we have 
$$\G(U)=\colim\limits_{\lambda}\G(U_{\lambda})\cong \tc\limits_{\lambda}\G(U_{\lambda})$$
Thus we can repeat the above computation replacing $\colim$ with $\tc$.
\end{proof}

\begin{proof} [Proof of Theorem \ref{Pi1}] By Theorem   \ref{vk} we know that the cosheaf  condition for the functor $U\mapsto \Pi_1(U)$ holds for coverings having only two members. By Proposition \ref{2}, for the same coverings the costack condition holds as well. Take  a groupoid $\sf G$ and consider the functor $F$ defined by
$$\F(U)=\hom_{\Grp}(\Pi_1(U),{\sf G})$$
Then $F$ satisfies both $sh(2)$ and $st(2)$ conditions. Hence for $F$  the conditions $sh(n)$ and $st(n)$ hold as well, thanks to  Lemma \ref{finite}. Since $\Pi_1$ commutes with filtered colimits, we can use Lemma \ref{filtered} to show that $U\mapsto \Pi_1(U)$ is both a cosheaf and a costack. 
\end{proof}


\begin{thebibliography}{999999999}
\bigskip

\bibitem{jajr} {\sc J. Ad\'amek} and {\sc J. Rosick\'y}. {\it Locally presentable and accessible categories}. LMS LNS 189. Cambridge University Press. 1994.

\bibitem{Brown} {\sc R. Brown}.  {\it Topology. A geometric account of general topology, homotopy types and the fundamental groupoid.} 
Ellis Horwood Limited. 1988.

\bibitem{br} {\sc J.-L. Brylinski}. {\it Loop spaces, characteristic classes and geometric quantization}. Progress in Mathematics. v. 107. Birkha\"user. Boston, Basel, Berlin. 1993.

\bibitem{gz} {\sc P. Gabriel} and {\sc M. Zisman}. {\it Calculus of fractions and homotopy theory}. Ergenbisse der Mathematik und ihre Grenzgebiete. Band 35. Springer-Verlag. Berlin, Heidelberg, New York, 1967.

\bibitem{May} {\sc P. May}. {\it  A concise course of Algebraic topology}. Chicago Lectures in Mathematics.

\end{thebibliography}
\end{document}